\documentclass[times, doublespace]{amsart}
\usepackage{amsmath, amssymb, amsthm,verbatim}
\usepackage[foot]{amsaddr}

\numberwithin{equation}{section}

\usepackage{setspace}
\usepackage[parfill]{parskip}

\usepackage[margin=1.5 in]{geometry}
\usepackage{tikz-cd}

\theoremstyle{plain}
\newtheorem{theorem}[equation]{Theorem}
\newtheorem{proposition}[equation]{Proposition}
\newtheorem{lemma}[equation]{Lemma}
\newtheorem{cor}[equation]{Corollary}

\theoremstyle{definition}
\newtheorem{definition}[equation]{Definition}

\newtheorem*{remark}{Remark}

\DeclareMathOperator{\Spec}{Spec}
\DeclareMathOperator{\Pic}{Pic}
\DeclareMathOperator{\cPic}{\mathbf{Pic}}
\DeclareMathOperator{\NS}{NS}

\DeclareMathOperator{\Jac}{Jac}
\DeclareMathOperator{\Prep}{Prep}

\DeclareMathOperator{\Alb}{Alb}

\DeclareMathOperator{\PGL}{PGL}

\DeclareMathOperator{\NT}{NT}
\DeclareMathOperator{\an}{an}

\renewcommand{\Im}{\operatorname{Im}}
\renewcommand{\div}{\text{div}}
\newcommand{\Tr}{\operatorname{Tr}}
\newcommand{\isom}{\xrightarrow{~\sim~}}
\newcommand{\hooklongrightarrow}{\lhook\joinrel\longrightarrow}

\newcommand{\imP}[1]{\Tr_{K/k}\left(\Pic^0(#1)\right)}

\newcommand{\trP}[1]{\operatorname{Pic}_{tr}(#1)}
\newcommand{\trPo}[1]{\operatorname{Pic}^0_{tr}(#1)}
\newcommand{\htau}{\widehat\tau}

\newcommand{\semipositive}{nef}

\newcommand{\C}{\mathbb{C}}
\newcommand{\R}{\mathbb{R}}
\newcommand{\Z}{\mathbb{Z}}

\newcommand{\Q}{\mathbb{Q}}
\newcommand{\F}{\mathbb{F}}
\newcommand{\A}{\mathbb{A}}
\newcommand{\cO}{\mathcal{O}}
\newcommand{\cX}{\mathcal{X}}
\newcommand{\cU}{\mathcal{U}}
\newcommand{\cY}{\mathcal{Y}}
\newcommand{\cM}{\mathcal{M}}
\newcommand{\calL}{\mathcal{L}}

\renewcommand{\P}{\mathbb{P}}

\begin{document}

\title{The arithmetic Hodge Index Theorem and rigidity of dynamical systems over function fields}

\author{Alexander Carney}

\date{\today}

\begin{abstract}
In one of the fundamental results of Arakelov's arithmetic intersection theory, Faltings and Hriljac (independently) proved the Hodge Index Theorem for arithmetic surfaces by relating the intersection pairing to the negative of the N\'eron-Tate height pairing. More recently, Moriwaki and Yuan--Zhang generalized this to higher dimension. In this work, we extend these results to projective varieties over transcendence degree one function fields. The new challenge is dealing with non-constant but numerically trivial line bundles coming from the constant field via Chow's $K/k$-trace functor. As an application of the Hodge Index Theorem, we also prove a rigidity theorem for the set of canonical height zero points of polarized algebraic dynamical systems over function fields. For function fields over finite fields, this gives a rigidity theorem for preperiodic points, generalizing previous work of Mimar, Baker--DeMarco, and Yuan--Zhang.
\end{abstract}

\maketitle

\section{Introduction}\label{intro}
The Hodge Index Theorem states classically that the divisor intersection pairing on an algebraic surface has signature $+1,-1,\dots,-1$. The corresponding result for line bundles on arithmetic surfaces, i.e. relative curves over the ring of integers of a number field, was proven independently by Faltings \cite{faltings} and Hriljac \cite{hriljac}, and is a fundamental result in Arakelov Theory. More recently, Moriwaki \cite{Moriwaki} extended this to higher dimensional arithmetic varieties, and Yuan and Zhang \cite{yz} proved a hodge index theorem for adelic metrized line bundles over $\overline{\Q}$.

In their work, Yuan and Zhang also conjectured that a similar result should hold over function fields. Here we prove their conjecture. Our theorem statement differs slightly from their conjecture, however, so as to be stated more directly and to avoid reliance on a non-canonical isogeny between Chow's function field $K/k$-trace and $K/k$-image. 

 Let $k$ be an algebraically closed field of arbitrary characteristic, and let $K=k(B)$ be the function field of $B$, a smooth, projective curve over $k$. Let $\pi:X\to\Spec (K)$ be a geometrically normal, geometrically integral, projective variety of dimension $n\ge1$. We will consider the group $\widehat{\Pic}(X)$ of adelic metrized line bundles on $X$ in the sense of \cite{Z95}; definitions will be recalled in Section \ref{adelicsubsection}.  Since an adelic metric can be specified, for example, by a line bundle on a model $\cX\to B$ of $X$, this setting also covers relative varieties fibered over $B$, in the same way that Yuan and Zhang's work over number fields encompasses Arakelov's setting of arithmetic varieties over the spectrum of the ring of integers of a number field. 
 
 Chow's $K/k$-trace functor $\Tr_{K/k}$ identifies the part of the Picard variety of $X$ which is defined over $k$, and the line bundles in $\Tr_{K/k}(\Pic^0(X))$ can all be given adelic metrics in a well-defined canonical way using isotrivial models over $B$. This construction is detailed in Section~\ref{tracesubsection}. Let $\Pic^{\tau}(X)$ be the numerically trivial subgroup of $\Pic(X)$. We prove the following result, with more detailed versions stated in Section~\ref{hitproof}:
 \begin{theorem}
 Let$M,N\in\Pic^{\tau}(X)$, and let $L_1,\dots L_{n-1}\in\Pic(X)$ be ample. There exist canonical metrics on $M$ and $N$ so that 
  \[\langle M,N\rangle_{L_1,\dots,L_{n-1}}:=\overline M\cdot\overline N\cdot\overline L_1\cdots\overline L_{n-1}\]
 is a well-defined bilinear pairing, independent of the choice of metrics on $L_1,\dots, L_{n-1}$. This extends to a symmetric $\R$-bilinear form on $\Pic^{\tau}(X)\otimes_{\Z}\R$ which is negative semidefinite with kernel 
\[ \Tr_{K/k}(\Pic^0(X))\otimes_{\Z}\R.\]
 \end{theorem}
 If one removes the function field trace (so that the kernel is trivial), this is the same result that Yuan--Zhang prove for number fields. It is straightforward to see that $\Tr_{K/k}(\Pic^0(X))\otimes_{\Z}\R$ is in the kernel. Thus, the main new difficulty is showing that numerically trivial adelic metrized line bundles which are non-constant must all come from isotrivial subgroups of the Picard group of $X$. In essence, all arguments of the proof must commute with the $K/k$-trace functor.
 
 \subsection{Arithmetic Dynamics}
Again let $X$ be a projective variety over a function field $K$. A polarized dynamical system $(f,L,q)$ is an endomorphism $f:X\to X$ along with an ample line bundle $L\in\Pic(X)$ such that $f^*L\cong L^{\otimes q}$ for some $q>1$. 
The set of preperiodic points of $f$ is defined as
\[\Prep(f):=\{x\in X(\overline K)|x\text{ has a finite forward orbit under }f\}.\]
Call and Silverman~\cite{CS} show that such a polarized endomorphism defines a canonical Weil height $\widehat h_f$. Here we show that $L$ can be given an \emph{admissible} metric $\overline L_f$ so that the height $h_{\overline L_f}$ defined by $\overline L_f$ via arithmetic intersections agrees with $\widehat h_f$ on $X(\overline K)$. The advantage to our definition is that $h_{\overline L_f}$ defines not only heights of points, but heights of subvarieties of $X$ as well. By applying the Hodge Index Theorem to compare the canonical heights defined by two different polarized dynamical systems, we prove the following rigidity theorem:

\begin{theorem}\label{introdynamicsthm}
Let $X$ be a projective variety defined over a transcendence degree one function field $K$ over any base $k$, and let $(f,L,q)$ and $(g,M,r)$ be two polarized dynamical systems on $X$. If the points with height zero under $h_{\overline L_f}$ and the points with height zero under $h_{\overline M_g}$ agree on a Zariski dense subset of $X(\overline K)$, then they are identical.
\end{theorem}

This is stated more generally in Section~\ref{dynamics}. When $k$ is a finite field, the Northcott property implies that the points with canonical height zero under $f$ are exactly the preperiodic points $\Prep(f)$, giving an immediate corollary. 

\begin{cor}\label{dynamicsthmfinite}
In the same setting as Theorem~\ref{introdynamicsthm} but with the additional hypothesis that $k=\F_q$, if $\Prep(f)\cap\Prep(g)$ is Zariski dense in $X(\overline K)$, then $\Prep(f)=\Prep(g)$.
\end{cor}

This was conjectured by Yuan--Zhang, and they prove a similar result over number fields.

Over general function fields fields the corollary does not hold, as not all canonical height zero points are preperiodic. The proofs differ as well, as while it is clear that the set $\Prep(f)$ does not depend on the choice of polarization $L$, this must be proven for canonical height zero points, and then the heights compared in a more indirect way. Even so, some limited things can be said.

\begin{cor}\label{bakercor}
Let $K$ be the function field of a smooth projective curve over any field $k$, and let $f$ and $g$ be two rational functions $\P^1_K\to\P^1_K$ which are not isotrivial. If $\Prep(f)$ and $\Prep(g)$ intersect on an infinite subset of $\P^1(\overline K)$, they are equal.
\end{cor}

This is a direct consequence of our theorem and a theorem of Baker~\cite{bakerisotrivial}. Chatzidakis and Hrushovski~\cite{chatzidakis1,chatzidakis2} prove results comparing preperiodic points and height zero points using a model-theoretic non-isotriviality condition in a much more general setting, but it is difficult to combine that result into a useful rigidity statement. This is discussed further in Section~\ref{future}.

\subsection{Outline of Paper and sketch of methods}
Definitions and basic properties of adelic metrized line bundles and Chow's $K/k$-image and trace are recalled in Section \ref{defs}. Additionally, this section includes technical lemmas, such as the existence of flat metrics, which will be needed throughout the paper. 

Our main Hodge Index Theorem, a classification of numerically trivial line bundles, and an $\R$-linear variant (Theorems~\ref{hodgeindex}, \ref{numericallytrivial}, and~\ref{rlinear}) are fully stated and proven in Section \ref{hitproof}. We begin with the case of $X$ being a curve. Decomposing adelic metrized line bundles into flat and vertical pieces, and addressing intersections of the vertical parts using the local Hodge Index Theorem of \cite[Theorem 2.1]{yz}, we reduce to the case of flat metrics. Then, following the methods of Faltings \cite{faltings} and Hriljac \cite{hriljac}, we relate the intersection pairing to the Ne\'ron-Tate height pairing on the Jacobian variety of $X$, and complete the result for curves using properties of heights on the Jacobian.

Next we prove the inequality part of Theorem \ref{hodgeindex} by induction on the dimension of $X$, using a Bertini-type theorem of Seidenberg \cite{Seidenberg} to find sections which cut out nice subvarieties of $X$. Along the way we prove a Cauchy-Schwarz inequality for this intersection pairing. Theorem~\ref{numericallytrivial} and the equality part of Theorem~\ref{hodgeindex} are then also proved by induction, where we again decompose into flat and vertical metrics and must show that the $K/k$-trace and image functors behave nicely when restricted to a subvariety. This is more difficult than the inequality, however. For the inequality, we write each metrized line bundle as a limit of model metrics, and prove the result for model metrics, thus getting the same inequality on their limit. We can write the same limit in the equality case, but we cannot assume that the same equality hypothesis holds for the model metrics, and must argue by other means. Finally, Theorem \ref{rlinear} is easily deduced from Theorem \ref{hodgeindex} and its proof.

Section~\ref{dynamics} proves the application of our result to polarized algebraic dynamical systems. We first describe and prove the existence of admissible metrics for a given polarized algebraic dynamical system, which generalize flat metrics, and give rise to canonical heights defined by intersections. This transforms the rigidity statement into a statement comparing two different admissible adelic metrized line bundles, which is proved using the Hodge Index Theorem.

Finally, Section~\ref{future} demonstrates corollaries of the main results proven here, and discusses what can still be said about preperiodic points over larger fields without the Northcott property.


\section{Preliminaries}\label{defs}

Here we introduce the definitions, basic properties, and lemmas which will be needed throughout the paper. The core theory used in this paper is built on local intersection theory as developed by Gubler \cite{G98,G07b}, Chambert-Loir \cite{C-L}, Chambert-Loir---Thuillier \cite{C-L-T}, and Zhang \cite{Z95}. More generally, one can find an introduction to Arakelov theory in~\cite{moriwakibook,langarakelov,soulelectures}.

\subsection{Metrized line bundles over local fields}\label{adelicsubsection}

Let $K$ be a complete non-Archimedean field with non-trivial absolute value $|\cdot|$. Denote the valuation ring of $K$ by 
\[K^{\circ}:=\{a\in K:|a|\le1\},\]
and its maximal ideal by
\[K^{\circ\circ}:=\{a\in K:|a|<1\},\]
so that $\widetilde K:=K^{\circ}/K^{\circ\circ}$ is the residue field. 

Let $X$ be a variety over $K$ and denote by $X^{\an}$ its Berkovich analytification as in \cite{Berk}. For $x\in X^{\an}$, write $K(x)$ for the residue field of $x$. A line bundle $L$ on $X$ has an analytification, denoted $L^{\an}$, as a line bundle on $X^{\an}$.

\begin{definition}
A \emph{continuous metric} $||\cdot||$ on $L$ consists of a $K(x)$-metric $||\cdot||_x$ on $L^{\an}(x)$ for every $x\in X^{\an}$, where this collection of metrics is continuous in the sense that for every rational section $s$ of $L$, the map $X^{\an}\to\R$ defined by $x\mapsto||s(x)||_x$ is continuous away from the poles of $s$. We call $L$ with a continuous metric a \emph{metrized line bundle} and denote this by $\overline L=(L,||\cdot||)$. For a fixed line bundle $L$, limits of metrics are taken with respect to the topology induced by the supremum norm.

An important example of a continuous metric is a \emph{model metric}: Let $\cX$ be a model of $X$ over $K^{\circ}$, i.e. a projective, flat, finitely presented, integral scheme over $\Spec K^{\circ}$ whose generic fiber $\cX_K$ is isomorphic to $X$, and let $\calL$ be a line bundle on $\cX$ whose generic fiber $\calL_K$ is isomorphic to $L$. Then we can define a continuous metric on $L$ by specifying that for any trivialization $\calL_{\cU}\isom\cO_{\cU}$ on an open set $\cU\subset\cX$ given by a rational section $\ell$,
we have $||\ell(x)||_x=1$ for any $x$ reducing to $\cU_{\widetilde K}$ in the reduction $\widetilde X$ over $\widetilde K$. 
\end{definition}

We now define several important properties and notations.

\begin{definition}\label{metrizedprops}
Let $\overline L=(L,||\cdot||)$ and $\overline M$ be metrized line bundles on $X$.
\begin{enumerate}
\item
A model metric is \emph{\semipositive{}} if it is given by a relatively nef line bundle on the corresponding model.
\item
Call both $\overline L$ and $||\cdot||$ \emph{\semipositive{}} if $||\cdot||$ is equal to a limit of \semipositive{} model metrics.
\item
$\overline L$ is \emph{arithmetically positive} if it is \semipositive{} and $L$ is ample.
\item
$\overline L$ is \emph{integrable} if it can be written as $\overline L=\overline L_1-\overline L_2$ with $\overline L_1$ and $\overline L_2$ \semipositive{}.
\item
$\overline M$ is \emph{$\overline L$-bounded} if there exists a positive integer $m$ such that $m\overline L+\overline M$ and $m\overline L-\overline M$ are both \semipositive{}.
\item
$\overline L$ is \emph{vertical} if it is integrable and $L\cong\cO_X$.
\item
$\overline L$ is \emph{constant} if it is isometric to the pull-back of a metrized line bundle on $\Spec K$.
\item
$\widehat{\Pic}(X)$ is defined to be the group of isometry classes of integrable metrized line bundles.
\end{enumerate}
\end{definition}

\begin{remark}
When we say a line bundle is relatively ample or nef, we always mean with respect to the structure morphism, here $\cX\to\Spec K^{\circ}$. A concise discussion of relative amplitude and nefness can be found in \cite{lazarsfeld}, Chapter 1.7.
\end{remark}

We also have a local intersection theory for metrized line bundles on $X$, due to~\cite{G98,G07}, and to~\cite{Z95} when $K$ has a discrete valuation. Let $Z$ be a $d$-dimensional cycle on $X$, let $\overline L_0,\dots,\overline L_d$ be integral metrized line bundles on $X$, and $\ell_0,\dots,\ell_d$ sections of each respectively such that 
\[\left(\bigcap_i|\div(\ell_i)|\right)\cap|Z|=\emptyset,\]
where $|Z|$ means the underlying topological space of the cycle $Z$. Then $Z$ has a local height $\widehat{\div}(\ell_0)\cdots\widehat{\div}(\ell_d)\cdot[Z]$ with the following properties:

\begin{enumerate}
\item
The local height is linear in $\widehat{\div}(\ell_i)$ and $Z$.
\item
For fixed sections, it is continuous with respect to the metrics.
\item
When $\overline L_i$ has a model metric given by $\calL_i$ on a common model $\cX$, the height is given by classical intersections:
\[\widehat{\div}(\ell_0)\cdots\widehat{\div}(\ell_d)\cdot[Z]=\div_{\cX}(\ell_0)\cdots\div_{\cX}(\ell_d)\cdot[\mathcal Z],\]where $\mathcal Z$ is the Zariski closure of $Z$ in $\cX$.
\item
If the support of $\div(\ell_0)$ contains no component of $Z$, there is a measure $c_1(\overline L_1)\cdots c_1(\overline L_d)\delta_Z$ on $X^{\an}$ due to \cite{C-L} which allows $\widehat{\div}(\ell_0)\cdots\widehat{\div}(\ell_d)\cdot[Z]$ to be computed inductively as 
\[\widehat{\div}(\ell_1)\cdots\widehat{\div}(\ell_d)\cdot[\div(\ell_0)\cdot Z]-\int_{X^{\an}}\log ||\ell_0(x)||_xc_1(\overline L_1)\cdots c_1(\overline L_d)\delta_Z.
\]
This notation is meant to suggest that $c_1(\overline L_i)$ should be thought of as the arithmetic version of the classical Chern form $c_1(L_i)$.
\item
If $L_0|_{Z_j}\cong\cO_{Z_j}$ and $c_1(L_1)\cdots c_1(L_d)\cdot[Z_j]=0$ for every irreducible component $Z_j$ of $Z$, then this pairing does not depend on the choice of sections, so we may simply write
\[\overline L_0\cdots\overline L_d\cdot Z=\widehat{\div}(\ell_0)\cdots\widehat{\div}(\ell_d)\cdot[Z].\]
\end{enumerate}
When $Z=X$, we typically omit $Z$ in all of the above notation.

By definition every integrable metric can be written as a limit of model metrics (with respect to the supremum norm).  Properties (3) and (4) above guarantee that intersections of integrable metrized line bundles are equal to the corresponding limits of intersections of models which approximate them.

\subsection{Adelic metrized line bundles}
We now move to the global theory, which is built from the theory of metrized line bundles over each localization, discussing first models and then adelic metrized line bundles. We return to the setting of the main theorems of this paper, where $k$ is any algebraically closed field, $B$ is a smooth projective curve over $k$, $K=k(B)$ is its function field, and $\pi:X\to\Spec(K)$ is a geometrically normal, geometrically integral, projective variety.

Let $\cX$ be a model for $X$, meaning that $\cX\to B$ is geometrically integral, projective, and flat, and the generic fiber $\cX_K$ is isomorphic to $X$. Given a geometrically integral subvariety $\cY$ of dimension $d+1$ in $\cX$ and line bundles $\calL_0,\dots,\calL_d$ on $\cX$ each with a respective section $\ell_0,\dots,\ell_d$ such that their common support has empty intersection with $\cY_K$, the arithmetic intersection pairing on $\Pic(\cX)$ is defined locally as
\begin{equation*}
\calL_0\cdots\calL_d\cdot\cY:=\widehat{\div}(\ell_0)\cdots\widehat{\div}(\ell_d)\cdot[\cY]\\:=\sum_{\nu}\left(\widehat{\div}(\ell_0)\cdots\widehat{\div}(\ell_d)\cdot[\cY]\right)_{\nu},
\end{equation*}
where $\nu$ ranges over the closed points (places) of $B$, and 
\[\left(\widehat{\div}(\ell_0)\cdots\widehat{\div}(\ell_d)\cdot[\cY]\right)_{\nu}\]
means the local intersection number after base-change to the complete field $K_{\nu}$. As the notation suggests, this does not depend on the choice of sections. Again we typically drop $\cY$ in the notation if $\cY=\cX$, and when $\cX$ is one-dimensional, we call $\widehat{\deg}(\calL_0):=\calL_0\cdot\cX$ the arithmetic degree of $\calL_0$.

\begin{remark}This arithmetic intersection theory for $\cX\to B$ is not the same as the classical intersection theory given by viewing $\cX$ as a variety over the field $k$. Instead, it matches Arakelov's arithmetic intersection theory~\cite{arakelov2,arakelov} except that in our function field setting there are no Archimedean places to consider, as $B$ is projective.
\end{remark}

Given a line bundle $L$ on $X$ we call a line bundle $\calL$ on $\cX$ a model for $L$ provided that $\calL_K\cong L$. For each place $\nu$ of $B$, completing with respect to $\nu$ induces a model over $K_{\nu}^{\circ}$ and a model metric $||\cdot||_{\calL,\nu}$ of $L^{\an}_{K_{\nu}}$ on $X^{\an}_{\nu}:=X^{\an}_{K_{\nu}}$.

\begin{definition}
The collection $||\cdot||_{\calL,\A}=\{||\cdot||_{\calL,\nu}\}_{\nu}$ of continuous metrics for every place $\nu$ of $B$ given by $(\cX,\calL)$ is called a \emph{model adelic metric} on $L$. More generally, an \emph{adelic metric} $||\cdot||_{\A}$ on $L$ is a collection of continuous metrics $||\cdot||_{\nu}$ of $L^{\an}_{K_{\nu}}$ on $X^{\an}_{\nu}$ for every place $\nu$, which agrees with some model adelic metric at all but finitely many places. A line bundle on $X$ with an adelic metric is called an \emph{adelic metrized line bundle}, and is denoted $\overline L=(L,||\cdot||_{\A})$. When the context is clear we will frequently drop adelic and simply write \emph{metrized line bundle}. For a fixed line bundle $L$, limits of adelic metrics are taken with respect to the topology induced by $\max_{\nu}||\cdot||_{sup}$, the maximum of the supremum norm on each fiber. Such a limit does not require fixing a single model $\cX$.
\end{definition}

We extend our local definitions of properties of metrized line bundles to the global case.
\begin{definition}
Let $\overline L$ be an adelic metrized line bundle.
\begin{enumerate}

\item
$\overline L$ is \emph{nef} if it is equal to a limit of model metrics induced by nef line bundles on models of $X$.

\item 
$\overline L$ is \emph{integrable} if it can be written as $\overline L=\overline L_1-\overline L_2$, where each $\overline L_i$ is nef.

\item
$\overline L$ is \emph{arithmetically positive} if $L$ is ample and $\overline L-\pi^*\overline N$ is nef for some adelic metrized line bundle $\overline N$ on $\Spec K$ with $\widehat{\deg}(\overline N)>0$.
\item
$\overline M$ is \emph{$\overline L$-bounded} if there exists a positive integer $m$ such that $m\overline L+\overline M$ and $m\overline L-\overline M$ are both \semipositive{}.
\item
$\overline L$ is \emph{vertical} if it is integrable and $L\cong\cO_X$
\item
$\overline L$ is \emph{constant} if it is isometric to the pull-back of a metrized line bundle on $\Spec K$.
\item
$\widehat{\Pic}(X)$ is defined to be the group of isometry classes of integrable metrized line bundles.

\end{enumerate}
\end{definition}

\begin{remark}
In the definition of arithmetically positive, we have thus-far only defined the arithmetic degree in the model case, but every adelic metrized line bundle on $\Spec K$ has a model metric, so we may use that definition. The definition is also resolved in the following material.
\end{remark}

\begin{remark}
To avoid confusion, note that the preceding definitions are not equivalent to requiring that the local property of the same name holds at every fiber. In fact, since relative amplitude (resp. nefness) holds if and only if the restriction to every fiber is ample (resp. nef), if a property holds in the global setting then the corresponding property holds locally at every place, but the converse is false. For example, if $\overline L_{\nu}$ is nef on $X_{\nu}$ for every place, each $\overline L_{\nu}$ can be written as a limit of nef models $\calL_{\nu,i}$ on $\cX_{\nu,i}$, but it may not be possible to assemble these into global models $\calL_i$ on models $\cX_i$ of $X$.
\end{remark}

Global intersections are defined similarly to the model case, except with the local metrics given explicitly by the adelic metric instead of induced by a model. Given a $d$-dimensional integral subvariety $Z$ of $X$, integrable adelic metrized line bundles $\overline L_0,\dots,\overline L_d$ with respective sections $\ell_0,\dots,\ell_d$ with empty common intersection with $Z$, their intersection is 
\begin{equation*}\overline L_0\cdots\overline L_d\cdot Z:=\widehat{\div}(\ell_0)\cdots\widehat{\div}(\ell_d)\cdot[Z]\\=\sum_{\nu}\widehat{\div}(\ell_0|_{X_{\nu}})\cdots\widehat{\div}(\ell_d|_{X_{\nu}})\cdot[Z|_{X_{\nu}}],
\end{equation*}
where again this is independent of the choice of sections. Summing the local induction formula at each place produces a global induction formula: letting $\ell_0$ be a rational section of $\overline L_0$ whose support does not contain $Z$, 
\begin{equation*}\overline L_0\cdots\overline L_d\cdot Z=\overline L_1\cdots\overline L_d\cdot (Z\cdot\div(\ell_0))\\-\sum_{\nu}\int_{X_{\nu}^{\an}}\log ||\ell_0(x)||_{\nu}c_1(\overline L_1,\nu)\cdots c_1(\overline L_d,\nu)\delta_Z|_{X_{\nu}}.
\end{equation*}

As before, we drop $Z$ when $Z=X$, and when $X$ is zero-dimensional, we call $\widehat{\deg}(\overline L_0):=\overline L_0\cdot X$ the arithmetic degree of $\overline L_0$.

As in the local case, we can always compute intersections of adelic metrized line bundles by approximating them with model metrics and computing the limit of the corresponding arithmetic intersections of the models.

\begin{definition}
An adelic metrized line bundle $\overline M$ on $X$ of dimension $n$ is called \emph{numerically trivial} if for any $\overline L_1,\dots,\overline L_n\in\widehat{\Pic}(X)$,
\[\overline M\cdot \overline L_1\cdots\overline L_n=0.\]
Call two adelic metrized line bundles \emph{numerically equivalent} if their difference is numerically trivial.
\end{definition}

\subsection{Flat metrics}\label{flatsubsection}
Adelic metrized line bundles with flat metrics form an especially nice class of adelic metrized line bundles. We will often be able to split a metrized line bundle into a bundle with a flat metric plus a vertical bundle, and then work with each of these separately, as flatness will tell us that these have trivial intersection.

\begin{definition}
Let $X$ be a projective variety over a complete field $K$, and let $\overline L$ be a metrized line bundle on $X$. Then $\overline L$ is \emph{flat} if for any morphism $f:C\to X$ of a projective curve over $K$ into $X$, we have $c_1(f^*\overline L)=0$ on the Berkovich analytification $C^{\an}$. If now $X$ is a projective variety over a global field and $\overline L$ an adelic metrized line bundle on $X$, call $\overline L$ flat provided it is flat at every place.
\end{definition}
Note that if $\overline L$ is flat, $L$ must be numerically trivial, as
\[\deg(L|_C)=\int_{C^{\an}}c_1(\overline L|_C)=0.\]

\begin{lemma}\label{flatlemma}
Let $L$ be a numerically trivial line bundle on a projective, normal variety $X$ over a global function field $K$. Then $L$ has a flat metric, which is unique up to constant multiple. 
\end{lemma}

\begin{remark}
When $X$ is a curve this lemma has a much simpler proof using linear algebra; see for example \cite[Theorem 1.3]{hriljac}. If $\cX\to B$ is a model for $X$ and $\cX_{\nu}$ is geometrically normal (so including every place with good reduction), then the flat metric at $\nu$ is simply that associated to closure of the point $L$ in $\Pic(\cX_{\nu})$.
\end{remark}

To prove the lemma in general, the following related notion will be useful.
\begin{definition}
Let $\overline L$ be a metrized line bundle on an abelian variety $A$ such that $L$ is algebraically trivial. We call $\overline L$ \emph{admissible} if $[2]^*\overline L\cong 2\overline L$.
\end{definition}

\begin{proof}[Proof of the lemma]
First suppose $A$ is an abelian variety. Then $L$ is algebraically trivial, and we have an isomorphism $\phi:[2]^*L\cong2L$. Take any metric $||\cdot||_1$ on $L$. Then Tate's limiting argument, as in \cite[Theorem 2.2]{Z95}, shows that
\[||\cdot||_n:=\phi^*[2]^*||\cdot||^{\frac12}_{n-1}\]
converges to an admissible metric $||\cdot||_0$ on $L$, and that further this is the unique admissible metric on $L$ up to constant multiples. 

Now let $C\to A$ be a smooth projective curve in $A$. After a translation, we can fix a point $x_0\in C(K)$ which maps to $0\in A$. By the universal property of the Jacobian, $C\to A$ factors through the Jacobian map $C\to\Jac(C)$ taking $x_0\to0$, and the pullback of $(L,||\cdot||_0)$ to $\Jac(C)$ is also admissible. Then by \cite[Remark 3.14]{G07b}, $c_1(L,||\cdot||_0)=0$, and hence $L$ has a flat metric. By taking the tensor product of this metric with the inverse of any other flat metric on $L$, uniqueness up to constant multiple is reduced to showing that $||1||$ is constant for any flat metric on $\cO_X$. Any two points on $X$ are connected by a curve; let $D$ be its normalization. Then $||1||$ is constant by local Hodge Index Theorem in dimension one at each place.

Now choose a point $x_0\in X(K)$ and recall the Albanase map $i:X\to \Alb(X)$ taking $x_0$ to $0$. Since $L$ is numerically trivial, we may replace it by a multiple and assume it is algebraically trivial. Then $L$ corresponds to a $K$ point $\xi$ of $\Pic_{red,X}^0=\Alb(X)^{\vee}$. By definition, $L$ is (isomorphic to) the Poincar\'e bundle $P$ on $\Alb(X)\times \Alb(X)^{\vee}$ restricted to $\Alb(X)\times\{\xi\}$, then pulled back through 
\[i\times \text{id}:X\times \Alb(X)^{\vee}\longrightarrow  \Alb(X)\times \Alb(X)^{\vee}.\]

$P|_{\Alb(X)\times\{\xi\}}$ is algebraically trivial, and hence has a flat metric. But the pullback of a flat metric is also flat, so this defines a flat metric for $L$.
\end{proof}

The reason we care about flat metrics is demonstrated by the following lemma and its corollary:

\begin{lemma}\label{flatzeromeasure}
Let $K$ be a complete non-Archimedean field, and $X\to\Spec K$ a geometrically connected, geometrically normal, projective variety of dimension $n$, with a flat metrized line bundle $\overline M$. Then given any integrable metrized line bundles $\overline L_1,\dots, L_{n-1}$ on $X$,
\[c_1(\overline M)c_1(\overline L_1)\cdots c_1(\overline L_{n-1})=0.\]
\end{lemma}

\begin{proof}
We show that 
\[\int_{X^{\an}}\log||\ell_n(x)||_xc_1(\overline M)\cdot c_1(\overline L_1)\cdots c_1(\overline L_{n-1})=0\]
    for every section $\ell_n$ of any metrized line bundle $\overline L_n$. Proceed by induction on $n$. Since any integral metrized line bundle can be written as a difference of arithmetically positive metrized line bundles and the measure is additive with respect to the metrized line bundles, we may assume that $\overline L_{n-1}$ is arithmetically positive without loss of generality. Further, by approximation, it suffices to treat the case where $\overline L_{n-1}$ is a model metric, induced by some ample line bundle $\calL$ on a model $\cX$ for $X$. By Seidenberg's Bertini theorem~\cite[Theorem 7']{Seidenberg}, $\calL$ has a section $s$ which cuts out a horizontal, geometrically integral, normal subvariety $\cY$. After base changing to a finite extension $K'$ of $K$, we may assume that this subvariety is geometrically normal. Since this extension merely scales the intersection number by $[K':K]$ it has no effect on the proof of this lemma. Let $Y$ be the generic fiber of $\cY$, and let $Z$ be $\div(\ell_n)$ restricted to $Y$.

We compute an intersection product in two different ways. First,
\begin{eqnarray*}
\overline M\cdot\overline L_1\cdots\overline L_n&=&\overline M|_Y\cdot\overline L_1|_Y\cdots\overline L_{n-2}|_Y\cdot\overline L_n|_Y\\
&=& \overline M|_Z\cdot\overline L_1|_Z\cdots\overline L_{n-2}|_Z-\int_{X^{\an}}\log||\ell_n(x)||_xc_1(\overline M)c_1(\overline L_1)\cdots c_1(\overline L_{n-2})\delta_Y\\
&=&\overline M|_Z\cdot\overline L_1|_Z\cdots\overline L_{n-2}|_Z,\\
\end{eqnarray*}
where the first equality follows from $\cY$ being horizontal, the second from the induction formula for local intersection numbers, and the third from the induction hypothesis. We now compute this in a different order:
\begin{eqnarray*}
\overline M\cdot\overline L_1\cdots\overline L_n&=&\overline M\cdot\overline L_1\cdots\overline L_{n-1}\cdot(\div(\ell_n))-\int_{X^{\an}}\log||\ell_n(x)||_xc_1(\overline M)c_1(\overline L_1)\cdots c_1(\overline L_{n-1})\\
&=&\overline M|_Z\cdot\overline L_1|_Z\cdots\overline L_{n-2}|_Z-\int_{X^{\an}}\log||\ell_n(x)||_xc_1(\overline M)c_1(\overline L_1)\cdots c_1(\overline L_{n-1}),\\
\end{eqnarray*}
where now the first inequality follows from the induction formula, and the second from $\cY$ being horizontal. Comparing the two equalities completes the proof.
\end{proof}

\begin{cor}\label{flatzerocor}
Let $\overline M$ be flat, $\overline N$ be vertical, and $\overline L_1,\dots,\overline L_{n-1}$ any integrable adelic metrized line bundles. Then
\[\overline M\cdot\overline N\cdot\overline L_1\cdots\overline L_{n-1}=0.\]
\end{cor}
\begin{proof}
Since $\overline N$ is vertical, $N=\cO_X$. Compute this intersection using the induction formula with the section $s=1$ of $\cO_X$:
\[\overline M\cdot\overline N\cdot\overline L_1\cdots\overline L_{n-1}=\overline M\cdot\overline L_1\cdots\overline L_{n-1}\cdot(\div (s))-\int_{X^{\an}}\log||1||_xc_1(\overline M)\cdot c_1(\overline L_1)\cdots c_1(\overline L_{n-1}).\]
The first term is zero since $\div(s)$ is empty, and the integral is zero by the preceding lemma.
\end{proof}

\subsection{Heights of points and subvarieties}
An important application of the intersection theory of adelic metrized line bundles is to define height functions. 

\begin{definition}
Let $\overline N\in\widehat{\Pic}(X)$. We define the \emph{height} of a point $x\in X(\overline K)$ by
\[h_{\overline N}(x):=\frac1{[K(x):K]}\overline N\cdot \tilde x,\]
where $\tilde x$ is the image of $x$ in $X$ via $X_{K(x)}\to X_K=X$.
\end{definition}

This is equal to the more well-known definition of the canonical height $\widehat{h}_f$, which does not require intersection theory~\cite{CS,heights}, but we use the above definition as it generalizes to define heights of subvarieties.

\begin{definition}
Let $d=\dim Y$. The \emph{height} of $Y$ with respect to $\overline N$ is defined to be
\[h_{\overline N}(Y):=\frac{\left(\overline N|_Y\right)^{d+1}}{(d+1)\left(N|_Y\right)^{d}}\]
and the \emph{essential minimum} of $Y$ with respect to $\overline N$ is
\[\lambda_1(Y,\overline N):=\sup_{\substack{U\subset Y \\ \text{open}}}\left(\inf_{x\in U(\overline K)} h_{\overline N|_Y}(x).\right)\]
\end{definition}

By the successive minima of Zhang \cite[Theorem 1.1]{Z95}, and proven in the function field setting by Gubler \cite[Theorem 4.1]{G07}, we can state the following.
\begin{proposition}\label{successiveminima}
When $\overline N$ is nef,
\[\lambda_1(Y,\overline N)\ge h_{\overline N}(Y)\ge0.\]
\end{proposition}

\subsection{Abelian varieties and Chow's $K/k$-trace and image}\label{tracesubsection}

Proofs of the existence and properties of the trace and image can be found in \cite{langab} and \cite{conradtrace}. Let $A$ be an abelian variety defined over $K$. The $K/k$-image $\left(\Im_{K/k}(A),\lambda\right)$ consists of an abelian variety $\Im(A)$ over $k$ and a surjective morphism 
\[\lambda:A\longrightarrow\Im_{K/k}(A)_K\]
with the following universal property: If $V$ is an abelian variety defined over $k$, and $\phi:A\to V_K$ a morphism, then $\phi$ factors through $\lambda$. Provided the fields $K$ and $k$ are clear, we will often drop the $K/k$ subscript and just write $\Im(A)$. 

The $K/k$-trace is $\left(\Tr_{K/k}(A),\tau\right)$ where $\Tr_{K/k}(A)$ is an abelian variety over $k$, and 
\[\tau:\Tr_{K/k}(A)_K\longrightarrow A\]
is universal among all morphisms from $k$-abelian varieties to $A$. Again we will often drop the $K/k$ when the fields are unambiguous. The image can be thought of as the largest quotient of $A$ that can be defined over $k$ and the trace the largest abelian subvariety that can be defined over $k$. This heuristic is literally true in characteristic zero, but in positive characteristic the trace map may have an infinitesimal kernel; see \cite[Section 6]{conradtrace}.

These constructions are dual to each other in the sense that 
\[\Tr(A^{\vee})=\Im(A)^{\vee},\]
and the image and trace are isogenous via the composition $\lambda\circ\tau$ (descended to the $k$-varieties). 

Given a morphism of abelian varieties $f:A\to B$, we get morphisms $f_{\Tr}:\Tr(A)\to\Tr(B)$ and $f_{\Im}:\Im(A)\to\Im(B)$ commuting with $\tau$ and $\lambda$. 

Now suppose $X$ is a geometrically normal projective variety over $K$ of dimension $n$, and assume that $K$ is large enough so that $X(K)$ is non-empty. We write $\cPic_X$ for the Picard scheme of $X$, representing the Picard functor on $X$. This scheme exists (i.e. the Picard functor is representable), and its reduced neutral component, denoted $\cPic^0_{red,X}$, is an abelian variety~\cite{Kleiman},~\cite[Lec. 236]{FGA}. Note that we do require the reduction, as $\cPic^0_X$ may fail to be reduced in positive characteristic. Write $\Pic(X)$ and $\Pic^0(X)$ for the abelian groups of $K$ points of $\cPic_X$ and $\cPic^0_{red,X}$, respectively.

We can then define $\Alb_X$, called the Albanese variety of $X$, to be the abelian variety dual to $\cPic^0_{red,X}$. Choosing a point $x_0\in X(K)$ fixes an Albanese morphism 
\[\iota:X\longrightarrow \Alb_X\]
taking $x_0$ to $0$, and then $(\Alb_X,\iota)$ uniquely satisfies the Albanese universal property: any morphism from $X$ to an abelian variety taking $x_0$ to zero must factor through $\iota$~\cite{Wittenberg08}.

We now have the language to differentiate between metrized line bundles defined over the constant field $k$ and those which are not. Define a group homomorphism
\[\htau_{K/k}:\Tr_{K/k}\left(\cPic^0_{red,X}\right)(k)\longrightarrow\widehat\Pic(X)\]
as follows.
First, by the duality of the $K/k$-trace and image,
\[\Tr_{K/k}\left(\cPic^0_{red,X}\right)(k)=\Pic^0\left(\Im_{K/k}(\Alb_X)\right)\]
Then we can map
\[\Pic^0\left(\Im_{K/k}(\Alb_X)\right)\hooklongrightarrow
\Pic\left(\Im_{K/k}(\Alb_X)\right)\longrightarrow
\Pic\left(\Im_{K/k}(\Alb_X)\times_kB\right),\]
where the map on the right is the pullback of projection onto the first factor. Since $\Im_{K/k}(\Alb_X)$ is defined over $k$, the fibered product $\Im_{K/k}(\Alb_X)\times_kB$ is a model for $\Im_{K/k}(\Alb_X)\times_kK$, and thus we get a map
\[\Pic\left(\Im_{K/k}(\Alb_X)\times_kB\right)\longrightarrow
\widehat\Pic\left(\Im_{K/k}(\Alb_X)_K\right)\]
given by taking model metrics. Finally, $X$ maps to $\Im_{K/k}(\Alb_X)_K$ via the Albanese map followed by the image map, and pulling this back gives
\[\widehat\Pic\left(\Im_{K/k}(\Alb_X)_K\right)\longrightarrow\widehat\Pic(X).\]
We can thus define $\htau_{K/k}$ as the composition of the above maps. While it took several steps to formally define $\htau_{K/k}$, it is very natural; if we define 
\[\phi:\widehat\Pic(X)\longrightarrow\Pic(X)\]
by forgetting the metric, then
\[\phi\circ\htau_{K/k}=\tau_{K/k},\]
the $K/k$-trace morphism (on field valued points), and the image of this composition lands in $\Pic^0(X)$. To simplify notation, we write $\Tr_{K/k}(\Pic^0(X))$ to mean the image of $\htau_{K/k}$ in $\widehat\Pic(X)$. By construction $\Tr_{K/k}(\Pic^0(X))$ is numerically trivial, as on every fiber $X_{\nu}$ this group restricts to $\Tr_{K/k}(\cPic_{red,X}^0)_{K_{\nu}}(k)$, which is algebraically trivial.


\section{Proof of Hodge Index Theorem}\label{hitproof}

\subsection{Statement of results}

Let $k$ be any algebraically closed field, let $B$ be a smooth projective curve over that field, and let $K=k(B)$ be the corresponding function field. Let $X$ be a geometrically normal projective variety over $K$ of dimension $n$, and assume that $K$ is large enough so that $X(K)$ is non-empty. Then choosing a point $x\in X(K)$ we may fix an Albanese morphism $\iota:X\to\Alb_X$. We impose these conditions on $X$ as well as this choice of Albanese morphism throughout the rest of the paper.

We can now state our main theorem:
\begin{theorem}[Arithmetic Hodge Index Theorem for function fields]\label{hodgeindex}
Let $\overline{M}$ be an integrable adelic $\Q$-line bundle on $X$ and $\overline{L}_1,\dots,\overline{L}_{n-1}$ nef adelic $\Q$-line bundles on $X$. Suppose if $n\ge2$ that $M\cdot L_1\dots L_{n-1}=0$ and each $L_i$ is big, or that $\deg M=0$ if $n=1$. Then
\[\overline{M}^2\cdot\overline{L}_{1}\dots\overline{L}_{n-1}\le0.\] 
Further, if every $\overline{L}_{i}$ is arithmetically positive, and $\overline M$ is $\overline L_i$-bounded for every $i$, then
\[\overline{M}^2\cdot\overline{L}_{1}\dots\overline{L}_{n-1}=0\]
if and only if 
\[\overline M\in\pi^*\widehat{\Pic}(K)_{\Q}+\Tr_{K/k}\left(\Pic^0(X)\right)_{\Q}.\]
When $n=1$ so that $X$ is a curve,
\[\overline M^2=-2h_{\NT}(M),\]
where $h_{\NT}$ is the N\'eron-Tate height on the Jacobian of $X$.
\end{theorem}

Note the important case that when $k$ is finite, $\Tr_{K/k}\left(\Pic^0(X)\right)_{\Q}$ is zero.

Call a metrized line bundle $\overline M$ on $X$ numerically trivial if 
\[\overline M\cdot \overline L_1\cdots \overline L_n=0\]
for every choice of metrized line bundles $\overline L_1,\dots,\overline L_n$. The classical Hodge Index Theorem says that the only divisors on a surface with zero self intersection are the numerically trivial divisors. We show that that is nearly, but not quite the case here:
\begin{theorem}\label{numericallytrivial}
The following three subgroups of $\widehat{\Pic}(X)$ are equal:
\begin{enumerate}
\item
The numerically trivial elements of $\widehat{\Pic}(X)_{\Q}$.
\item 
The set of $\overline M\in\widehat{\Pic}(X)_{\Q}$ such that the height $h_{\overline M}$ is identically zero on $X(\overline K)$.
\item
$\pi^*\widehat{\Pic}^0(K)_{\Q}+\imP{X}_{\Q},$ where $\widehat{\Pic}^0(K)_{\Q}$ is defined to be the elements of $\widehat{\Pic}(K)_{\Q}$ with arithmetic degree zero.
\end{enumerate}
\end{theorem}

Define $\Pic^{\tau}(X)$ to be the group of isomorphism classes of numerically trivial line bundles on $X$. We define a pairing on $\Pic^\tau(X)$ to give an $\R$-linear version of Theorem~\ref{hodgeindex}. Let $M,N\in\Pic^{\tau}(X)_{\R}$, and let $L_1,\dots,L_{n-1}\in\Pic(X)_{\Q}$ be nef. Then define a pairing by
\[\langle M,N\rangle_{L_1,\dots,L_{n-1}}:=\overline M\cdot\overline N\cdot\overline L_1\cdots\overline L_{n-1},\]
using any choice of flat metrics on $M$ and $N$, and any choice of metrics on $L_i$. By Lemma 5.19 of \cite{yz}, (proven as a simple consequence of Lemma~\ref{flatzeromeasure} here) this pairing does not depend on the choice of metric.

\begin{theorem}\label{rlinear}
For any $M\in\Pic^{\tau}(X)_\R$ and nef $L_1,\dots,L_{n-1}\in\Pic(X)_{\Q}$, 
\[\langle M,M\rangle_{L_1,\dots,L_{n-1}}\le0.\]
Further, if every $L_i$ is ample, then equality holds if and only if $M\in\imP{X}_{\R}$.
When $X$ is a curve, 
\[\langle\cdot,\cdot\rangle=-2\langle\cdot,\cdot\rangle_{\NT},\]
where $\langle\cdot,\cdot\rangle_{\NT}$ is the N\'eron-Tate height pairing on the Jacobian of $X$.
\end{theorem}

These results are proven over the next three sections, with the bulk of the work going into proving Theorem~\ref{hodgeindex}, and then Theorems~\ref{numericallytrivial} and~\ref{rlinear} following as corollaries.

\subsection{Curves}

We begin when $X$ is a curve. Here we can work directly in $\widehat{\Pic}(X)$ as opposed to $\widehat{\Pic}(X)_{\Q}.$
Then the theorem discusses the self-intersection $\overline M^2$ when $\deg M=0$. By Lemma \ref{flatlemma}, $M$ has a flat metric $\overline M_0=(M,||\cdot||_0)$.

Let $\overline N$ be the vertical line bundle defined by 
\[\overline M=\overline M_0+\overline N.\]
Since $\overline M_0$ is flat, $\overline M_0\cdot \overline N=0$ so that 
\[\overline M^2=\overline M_0^2+\overline N^2=\overline M_0^2+\sum_{\nu}\overline N^2_{\nu},\]
where $\overline N_{\nu}$ is the restriction of of $\overline N$ to $X_{\nu}:=X\otimes_KK_{\nu}$ for each place $\nu$ of $K$ (i.e closed point of $B$).
Now $\overline N_{\nu}^2\le0$ with equality if and only if $\overline N_{\nu}$ is constant by the local hodge index theorem, \cite{yz} Theorem 2.1. 
Hence 
\[\sum_{\nu}\overline N^2_{\nu}\le0,\]
with equality if and only if $\overline N\in \pi^*\widehat{\Pic}(K)$. 

Next, we consider $\overline M_0^2$. Since $M$ has degree zero, it corresponds naturally to a $K$-point on the Jacobian, $\Jac_X$, of $X$. Given any two points $P,Q\in\Jac_X(K)$, let $L_P$ and $L_Q$ be the corresponding algebraically trivial line bundles on $X$. These each have a flat metric, $\overline L_P$ and $\overline L_Q$, unique up to constant metric, and thus we get a well defined symmetric bilinear pairing
\[(P,Q)\mapsto -\overline L_P\cdot \overline L_Q\]
on $\Jac_X(K)$, as the intersection does not depend on the choice of flat metric. As is noted in \cite{faltings} and \cite{hriljac} in the arithmetic setting, this pairing is exactly the N\'eron-Tate height pairing. Then the Shioda-Tate Theorem \cite[Theorem 7]{Shioda} states that this pairing descends to a positive definite pairing on $\Jac_X(K)_{\Q}/\Tr_{K/k}(\Jac_X)(k)$, and
\[\overline M^2=-2h_{\text{NT}}(M).\]

Thus we are done, provided that we verify that our map $\htau_{K/k}$ always produces elements of $\widehat\Pic(X)$ with flat metrics, so that our pairing on $\Tr_{K/k}(\Pic^0(X))$ matches that considered by Shioda. But this is clear: $\Tr_{K/k}(\Pic^0(X))$ is numerically trivial, so in particular has trivial intersection with all vertical line bundles, making it flat.

As this also extends $\R$-linearly, this completes the proof of Theorems ~\ref{hodgeindex} and~\ref{rlinear} in dimension 1, and Theorem~\ref{numericallytrivial} in dimension one when $\overline M$ is flat. We now complete the proof of Theorem~\ref{numericallytrivial} in dimension 1.

Suppose that $\overline M$ is given as a model metric by some $\cM$ on a model $\cX\to B$ for $X$. The general case follows from this case by taking limits of such metrics. If $\overline M$ is vertical, then a section of $\cM$ corresponds to a sum of components of fibers of $\cX\to B$, and $h_{\overline M}(x)$ is the intersection of this sum with the closure of $x$ in $\cX$. This is constant for all $x\in X(\overline K)$ if and only if the sum consists only of a constant multiple of whole fibers $\cX_{\nu}$, in which case $\overline M=\pi^*\overline N$ for some $\overline N\in\widehat{\Pic}(K)$, and then the height is the arithmetic degree of $\overline N$.

Finally, we consider heights $h_{\overline M}$ where $\deg M\ne0$. Suppose $\deg M>0$. Then $M$ is ample, and $h_{\overline M}$ is unbounded, thus non-constant. Replacing $\overline M$ with $-\overline M$ covers $\deg M<0$.

\subsection{Inequality and Cauchy-Schwarz}

We now prove the inequality part of Theorem \ref{hodgeindex} by induction on $n=\dim X$, and get a version of the Cauchy-Schwarz inequality as a corollary. As in \cite[Section 3.3 assumption (2)]{yz}, we may assume that each $\overline{L_i}$ is arithmetically positive (instead of just big) by a limiting argument. Thus we may assume $L_i$ is ample.

Since $\overline M$ and each $\overline L_i$ can be approximated by model metrics, it suffices to prove 
\[\cM^2\cdot \calL_1\cdots\calL_{n-1}\le0,\]
under the assumption that $\cM$ and every $\calL_i$ are line bundles on a model $\cX$ for $X$, that $\calL_i$ is ample with respect to $k$, and that the intersection $\cM_K\cdot (\calL_{1})_K\cdots(\calL_{n-1})_K$ on the generic fiber is zero. 

Replacing $\calL_1$ by a positive tensor power if necessary, we may assume it is very ample. Then by a Bertini-type result of Seidenberg \cite[Theorem 7']{Seidenberg}, a generic section of $\calL_1$ cuts out an integral normal subvariety $\cY$ of $\cX$, and we may further stipulate that $\cY$ is horizontal. Then
\[
\cM^2\cdot \calL_1\cdots\calL_{n-1}=\cM|^2_{\cY}\cdot \calL_1|_{\cY}\cdots\calL_{n-2}|_{\cY}.
\]
This reduces the problem to a lower dimension, but we require that $\cY_K$ have a $K$ point to conclude the result by induction. This is certainly true if we replace $K$ with a finite extension $K'$, or equivalently replace $B$ with a finite cover. Since intersection numbers simply scale by $[K':K]$ and the subgroup $\widehat\Pic(K)+\Tr_{K/k}\left(\Pic^0(X)\right)$ is equal to $\widehat\Pic(K')+\Tr_{K'/k}\left(\Pic^0(X_{K'})\right)$ intersected with $\widehat\Pic(X)$, such a base change is permissible.

Given $M\in\Pic^{\tau}(X)_\R$, we can write it as an $\R$-linear combination of numerically trivial line bundles on $X$, and each has a flat metric by Lemma~\ref{flatlemma}. Then the inequality of Theorem~\ref{hodgeindex} immediately implies the inequality of Theorem~\ref{rlinear} when every $L_i$ is big. If $L_i$ is merely nef, choose any ample line bundle $A$ and $\epsilon>0$, and then $L_{i_\epsilon}:=L_i+\epsilon A$ is big. Thus the inequality holds with $L_{i,\epsilon}$ replacing $L_i$, and taking the limit as $\epsilon\to0$, it holds in general.

As a corollary, we have the following Cauchy-Schwarz inequality:

\begin{cor}\label{cauchyschwarz}
Let $\overline M,\overline N$ be two integral adelic line bundles on $X$, and let $\overline L_1,\dots,\overline L_{n-1}$ be nef adelic line bundles on $X$ such that
\[M\cdot L_1\cdots L_{n-1}=N\cdot L_1\cdots L_{n-1}=0.\]
Then
\[\left(\overline M\cdot\overline N\cdot \overline L_1\cdots\overline L_{n-1}\right)^2\le\left(\overline M^2\cdot \overline L_1\cdots\overline L_{n-1}\right)\left(\overline N^2\cdot \overline L_1\cdots\overline L_{n-1}\right).\]
\end{cor}
\begin{proof}
This follows from the inequality part of the Hodge Index Theorem proven above, and from the standard proof of the Cauchy-Schwarz inequality using the (negative semi-definite) inner product
\begin{equation*}\langle M,N\rangle_{\overline L_1,\dots,\overline L_{n-1}}:=\overline M\cdot\overline N\cdot \overline L_1\cdots\overline L_{n-1}.\qedhere\end{equation*}
\end{proof}

\subsection{Equality}
Proceeding now to the equality part of Theorem \ref{hodgeindex}, we assume that each $\overline L_i$ is arithmetically positive, that $\overline M$ is $\overline L_i$-bounded for all $i$, and that
\[\overline M^2\cdot \overline L_1\cdots\overline L_{n-1}=0.\] 
Note that as a consequence of the Cauchy-Schwarz inequality above, the set of metrized line bundles $\overline M$ satisfying these properties forms a group via tensor products.

By \cite[Lemma 3.7]{yz} (this uses the fact that $\overline L_i$ is arithmetically positive), $M$ is numerically trivial on $X$. Thus it has a flat metric; let $\overline M_0=(M,||\cdot||)$ be flat. Then, similar to the curve case, $\overline N:=\overline M-\overline M_0$ is vertical, and 
\[\overline M^2\cdot \overline L_1\cdots\overline L_{n-1}=\overline M_0^2\cdot \overline L_1\cdots\overline L_{n-1}+\overline N^2\cdot \overline L_1\cdots\overline L_{n-1}.\]
The inequality part of the hodge index theorem guarantees that both terms on the right are zero, and then by the local hodge index theorem at every place occurring in $\overline N$, we have $\overline N\in\widehat{\Pic}(K)_{\Q}$. Hence we are reduced to proving the statement in the flat metric case $\overline M=\overline M_0$.

We again replace $\overline L_1$ by a positive multiple to assume that $L_1$ is very ample, then apply Seidenberg's Bertini Theorem to conclude that $(L_1)_{\overline K}$ has a section $s$ which cuts out an integral, normal subvariety $Y$. Such $Y$ is defined over some finite extension of $K'/K$, and thus after a base change from $K$ to $K'$, we may assume that $L_1$ has a section which cuts out a geometrically integral and geometrically normal subvariety $Y$. As in the proof of the inequality, this finite extension merely scales the intersection numbers by a positive factor. We thus continue writing $K$, assuming it has been made large enough, to avoid excessive additional notation.

\begin{lemma}\label{bertinisection}
If $\overline M$ is flat, and $Y$ is a geometrically normal subvariety of $X$, then
\[\overline M|_Y^2\cdot\overline L_2|_Y\cdots\overline L_{n-1}|_Y=0.\]
\end{lemma}
\begin{proof}
By the induction formula of Chambert-Loir \cite{C-L}, recalled here in Section~\ref{adelicsubsection},
\[
\overline M^2\cdot\overline L_1\cdots \overline L_{n-1}
=\overline M|_{Y}^2\cdot \overline L_2|_{Y}\cdots\overline L_{n-1}|_{Y}\\-\sum_v\int_{X^{\an}_v}\log||s||_vc_1(\overline M)^2c_1(\overline L_2)\cdots c_1(\overline L_{n-1}).
\]
Since $\overline M$ is flat, all the integrals are zero.
\end{proof}

Thus we may assume
\[\overline M|_Y\in\pi^*\widehat{\Pic}(K)_{\Q}+\imP{Y}_{\Q}\]
by induction.

Write $\overline M|_Y=\overline M'+\pi^*\overline M_1$, with $\overline M'\in\imP{Y}_{\Q}$ and $\overline M_1\in\widehat{\Pic}(K)_{\Q}$. Then define $\overline M_2=\overline M-\overline \pi^*\overline M_1$.
Since $M$ is numerically trivial, replacing $\overline M$ by a positive integer multiple if necessary, we may further assume $M$ is algebraically trivial, and then that $M_1,M_2\in\Pic^0(X)$.

As noted earlier, if we drop the metric structure the map $\htau$ is simply the $K/k$-trace map on field valued points. The following lemma then proves that $M_2|_Y=M'$ lifts via the pullback of $Y\hookrightarrow X$ to an element of $\imP{X}_{\Q}$.

 \begin{lemma}\label{traceint}
 Let $f:A\to B$ be a morphism of abelian varieties defined over $K$. In the commutative diagram
 \[
 \begin{tikzcd}
 \Tr(A)(k)_{\Q}\arrow{r}{\tau_A}\arrow{d}{f_{\Tr}}&A(K)_{\Q}\arrow{d}{f}\\
 \Tr(B)(k)_{\Q}\arrow{r}{\tau_B}&B(K)_{\Q}
 \end{tikzcd}
 \]
$(f\circ \tau_A)\left(\Tr(A)(k)_{\Q}\right)$ is equal to $f(A(K)_{\Q})\cap\tau_B\left(\Tr(B)(k)_{\Q}\right)$. 
 
 \end{lemma}

 \begin{proof}
 To shorten notation, we will drop writing the map $\tau_A$ and consider $\Tr(A)(k)$ directly as a subgroup of $A(K)$ (and similarly for $B$).
 First reduce to the case where $f$ is surjective: let $B'$ be the image of $f$, an abelian subvariety of $B$. By Poincar\'e reducibility, $B$ is isogenous to $B'\times B''$, for some abelian variety $B''$. Then $\Tr(B)$ is isogenous to $\Tr(B')\times \Tr(B'')$, and the intersection of $\Tr(B')(k)\times \Tr(B'')(k)$ with $B'(K)$ is just $\Tr(B')(k)$. 
 
 Now assume $f$ is surjective. We can find abelian subvarieties $A'\subset A$ and $B'\subset B$, and abelian varieties $A'',B''$ such that $A$ is isogenous to $A'\times A''$, $\Tr(A')=\Tr(A)$, and $\Tr(A'')=0$, and similarly for $B$. Then $f$ induces a surjection $A'\twoheadrightarrow B'$, but $A'$ is isogenous to $\Tr(A')_K$, and $B'$ is isogenous to $\Tr(B')_K$, so we get a surjection $\Tr(A)(k)\twoheadrightarrow\Tr(B)(k)$, proving the lemma.
\end{proof}

Hence we may lift $\overline M_2|_Y$ to an element $\overline M_2'\in \imP{X}_{\Q}$,
 and we must have 
\[\overline M_2-\overline M_2'\in\ker\left(\widehat{\Pic}(X)\longrightarrow\widehat{\Pic}(Y)\right).\]
Since $\Pic^0(X)\to\Pic^0(Y)$ has finite kernel~\cite[Remark 9.5.8]{Kleiman}, replacing $M$ with a positive integer multiple, we may assume $M_2-M_2'=\mathcal O_X$ and thus $\overline M_2-\overline M_2'$ is vertical. Additionally, by the Cauchy-Schwarz inequality, Corollary \ref{cauchyschwarz},
\[\left(\overline M_2-\overline M_2'\right)^2\cdot \overline L_1\cdots \overline L_{n-1}=\left(\overline M-\pi^*\overline M_1-\overline M_2'\right)^2\cdot \overline L_1\cdots \overline L_{n-1}=0,\]
so that by the local Hodge Index Theorem the metric must be constant at each place and $\overline M_2-\overline M_2'\in\pi^*\widehat{\Pic}(K)_{\Q}$. Note that the local Hodge Index Theorem requires the hypothesis that $\overline M$ be $\overline L_i$-bounded. This means that 
\[\overline M=\left(\pi^*\overline M_1+\overline M_2-\overline M_2'\right)+\overline M_2'\in\pi^*\widehat{\Pic}(K)_{\Q}+\imP{X}_{\Q}.\]

This proves that when $\overline M$ is $\overline L_i$ bounded and $\overline L_i$ is arithmetically positive for all $i$, then 
\[\overline{M}^2\cdot\overline{L}_{1}\cdots\overline{L}_{n-1}=0\]
if and only if $\overline M\in\pi^*\widehat{\Pic}(K)_{\Q}+\imP{X}_{\Q}$, completing the proof of Theorem~\ref{hodgeindex}.

Since we can cut out any point in $X(\overline K)$ by sections of line bundles, (1) and (2) of Theorem~\ref{numericallytrivial} are equal. By the Hodge Index Theorem, (1) is a subgroup of $\pi^*\widehat\Pic(K)_\Q+\imP{X}_\Q$. As noted earlier, all of $\imP{X}$ is numerically trivial. It is clear that the numerically trivial subgroup of $\pi^*\widehat\Pic(K)$ is exactly $\pi^*\widehat\Pic^0(K)$ for a curve, and the above induction argument establishes this in higher dimension as well, proving Theorem~\ref{numericallytrivial}.

To prove the equality of Theorem~\ref{rlinear}, again split $M\in\Pic^\tau(X)_\R$ into an $\R$-linear combination of numerically trivial line bundles. Using the inequality, the equality can be proven for each of these individually. $L_1$ is ample, so Lemma~\ref{bertinisection} applies, and then by the induction hypothesis $M|_Y\in\imP{Y}_\R$. Then by Lemma~\ref{traceint} we conclude $M\in\imP{X}_\R$.

Finally, we note that the above arguments also prove the equality part of Theorem \ref{rlinear}: Given $M\in\Pic^{\tau}(X)_{\R}$ and $L_i$ nef, $M$ has a $\R$-linear sum of flat metrics $\overline M$ as proven above, and each $L_i$ can be extended to a nef adelic metrized bundle $\overline L_i$. Lemma \ref{bertinisection} works just the same in this $\R$-linear setting, and then by induction, we can assume $\overline M|_Y\in \imP{Y}_{\R}$. Lemma \ref{traceint} is also the same in the $\R$-linear instead of $\Q$-linear setting, so that $\overline M\in \imP{Y}_{\R}$ as desired.


\section{Algebraic Dynamical Systems}\label{dynamics}

As before, $K$ is the function field of a smooth projective curve $B$ over $k$, and let $X$ be a projective variety over $K$. Suppose $(X,f,L)$ and $(X,g,M)$ are two polarized dynamical systems on $X$, so that $f$ and $g$ are endomorphisms of $X$, and $L$ and $M$ are ample line bundles such that $f^*L\cong L^q$ and $g^*M\cong M^r$ for some $q,r>1$. 

\begin{remark}
If $X$ is not normal, we may replace $X$ by its normalization $\psi:X'\to X$, replace $f$ by the normalization $f':X'\to X'$ of $f\circ \psi$, and replace $L$ by $L'=\psi^*L$ to get a new polarized algebraic dynamical system $(X',f',L')$ with $\Prep(f')=\psi^{-1}\Prep(f)$, and similarly for $(X,g,M)$. By first replacing $K$ with an extension if necessary, we may further assume that the normalization is geometrically normal. Hence from here on out we assume without loss of generality that $X$ is geometrically normal.
\end{remark}

Our main goal in this section is to prove a comparison theorem for the points with dynamical height 0 under $f$ and $g$, with an important corollary comparing the preperiodic points of $f$ and $g$ when $k$ is a finite field. We begin with general properties of polarized algebraic dynamical systems, then define the particular arithmetic dynamical heights involved before stating the theorem.

\subsection{An $f^*$-splitting of the N\'eron-Severi sequence}
We first show that the projection from $\Pic(X)$ onto the N\'eron-Severi group has a unique $f^*$ equivariant section.

The pullback $f^*$ preserves the exact sequence
\[0\longrightarrow\Pic^0(X)\longrightarrow\Pic(X)\longrightarrow\NS(X)\longrightarrow0,\]
defining the N\'eron-Severi group $\NS(X)$, and the N\'eron-Severi Theorem~\cite[Exp. XII, Thm 5.1, p. 650]{SGA6} tells us that $\NS(X)$ is a finitely generated $\Z$-module. For arbitrary $k$, the $\Z$-module $\Pic^0(X)$ need not be finitely generated, but by the Lang--N\'eron Theorem~\cite{langneron},
\[\Pic^0(X)\bigg/\Tr_{K/k}\Pic^0(X)\cong\Pic^0(X)\bigg/\Pic^0\left(\Im_{K/k}(\Alb(X))\right)\]
is a finitely generated $\Z$-module. 
To shorten our notation, define
\[\trPo{X}:=\Pic^0(X)\bigg/\Tr_{K/k}\Pic^0(X),\]
\[\trP{X}:=\Pic(X)\bigg/\Tr_{K/k}\Pic^0(X),\]
so that we have an exact sequence of finite-dimensional $\C$-vector spaces

\[0\longrightarrow\trPo{X}
_{\C}\longrightarrow\trP{X}
_{\C}\longrightarrow\NS(X)_{\C}\longrightarrow0,\]
which is also an exact sequence of $f^*$-modules.

\begin{lemma}\label{eigenvalues}
The operator $f^*$ is semisimple on $\trPo{X}_{\C}$ 
with eigenvalues of absolute value $q^{1/2}$, and is semisimple on $\NS(X)$ with eigenvalues of absolute value $q$.
\end{lemma}

\begin{proof}
As usual let $n=\dim X$. By the classical Hodge Index Theorem~\cite[Exp. XIII, Cor 7.4]{SGA6}, we can decompose $\NS(X)_{\R}$ as 
\[\NS(X)_{\R}:=\R L\oplus P(X),\quad P(X):=\left\{\xi\in\NS(X)_{\R}:\xi\cdot L^{n-1}=0\right\},\]
and define a negative definite pairing on $P(X)$ by
\[\langle\xi_1,\xi_2\rangle:=\xi_1\cdot\xi_2\cdot L^{n-2}.\]
The projection formula for intersection numbers applied to $L^n$ gives us $\deg f=q^n$, and then applied to this pairing, we have
\[\langle f^*\xi_1,f^*\xi_2\rangle=q^2\langle\xi_1,\xi_2\rangle.\]
Hence $\frac1qf^*$ is orthogonal with respect to this pairing, and $\frac1qf^*$ is diagonalizable on $\NS(X)_{\C}$ with eigenvalues all of absolute value 1.

On $\Pic^0(X)_{\R}$ we can define a pairing as follows: for $\xi_1,\xi_2\in\Pic^0(X)_{\R}$, let $\overline{\xi}_1$, $\overline{\xi}_2$ be flat metrized extensions, and let $\overline L$ be any integrable adelic line bundle extending $L$. Then define
\[\langle\xi_1,\xi_2\rangle:=\overline{\xi}_1\cdot\overline{\xi}_2\cdot\overline L^{n-1}.\]
It follows from Corollary~\ref{flatzerocor} that this pairing does not depend on the choice of metrics, and Theorem \ref{hodgeindex} establishes that it is negative definite. Since $\Tr_{K/k}\Pic^0(X)$ is numerically trivial, this pairing descends to $\trPo{X}_{\R}$.

Again applying the projection formula,
\[(f^*\overline{\xi}_1)\cdot (f^*\overline{\xi}_2)\cdot(f^*\overline L)^{n-1}=q^n(\overline{\xi}_1\cdot\overline{\xi}_2\cdot\overline L^{n-1}),\]
since each $f^*\overline{\xi}_i$ is still flat. We may also replace $f^*\overline L$ by $\overline L^q$ because the pairing is independent of the choice of metric on $L$, and have
\[\langle f^*\xi_1,f^*\xi_2\rangle=q\langle\xi_1,\xi_2\rangle.\]
Hence, $q^{-\frac12}f^*$ is orthogonal on $\trPo{X}_{\R}$ with respect to the negative of this pairing, making it diagonalizable with eigenvalues of absolute value 1 as a transformation on $\trPo{X}_{\C}$.
\end{proof}

By the theorem, 
\[0\longrightarrow\trPo{X}_{\C}\longrightarrow\trP{X}_{\C}\longrightarrow\NS(X)_{\C}\longrightarrow0\]
has a unique splitting as $f^*$-modules by a section
\[\ell_f:\NS(X)_{\C}\longrightarrow\trP{X}_{\C}.\]
Let $P,Q\in\Q[T]$ be the minimal polynomials of $f^*$ on $\trPo{X}_{\Q}$ and $\NS(X)_{\Q}$ respectively. 
Because the eigenvalues of $f^*$ are different on $\trPo{X}_{\Q}$ and $\NS(X)_{\Q}$, we see that $P$ and $Q$ are coprime, and $PQ$ is the minimal polynomial of $f^*$ on $\trP{X}_{\Q}$. Define
\[\Pic_{tr,f}(X)_{\Q}:=\ker Q(f^*)|_{\trP{X}_{\Q}}\]
and then this splitting can be given over $\Q$ as
\[\ell_f:\NS(X)_{\Q}\isom\Pic_{tr,f}(X)_{\Q}\hooklongrightarrow\trP{X}_{\Q}.\]

\subsection{Admissible metrics}

\begin{theorem}\label{admissible}
The projection $\widehat{\Pic}(X)_{\Q}\to\Pic(X)_{\Q}$ has a unique section $M\mapsto \overline M_f$ as $f^*$-modules, satisfying:
\begin{enumerate}
\item If $M\in\Pic^0(X)_{\Q}$ then $\overline M_f$ is flat.
\item If $M\in\Pic_{tr,f}(X)_{\Q}$ is ample then $\overline M_f$ is nef.
\end{enumerate}
Adelic metrized line bundles of the form $\overline M_f$ are called $f$-admissible.
\end{theorem}

\begin{proof}
Define $\widehat{\Pic}(X)'$ to  be the group of adelic line bundles on $X$ with continuous (but not necessarily integrable) metrics. This contains $\widehat{\Pic}(X)$. We will show that the projection $\widehat{\Pic}(X)'_{\Q}\to\Pic(X)_{\Q}$ has a unique section, and then that properties 1 and 2 of the theorem hold for this section. Since $\Pic^0(X)_{\Q}$ and the ample elements of $\Pic_{tr,f}(X)_{\Q}$ generate $\Pic(X)_{\Q},$ the section does in fact produce integrable metrics, proving the theorem.

The kernel of the projection $\widehat{\Pic}(X)'_{\Q}\to\Pic(X)_{\Q}$ is 
\[ D(X)=\widehat{\Pic}(K)_{\Q}\bigoplus_vC(X_v^{\an}),\]
where $C(X_v^{\an})$ is the ring of continuous $\R$-valued functions on $X_v^{\an}$, via the association $||\cdot||_v\to-\log ||1||_v$. Recall that $R=PQ$ was defined to be the minimal polynomial of $f^*$ on $\Pic(X)_{\Q}$ and now consider the action of $R(f^*)$ on $D(X)$.

\begin{lemma}
$R(f^*)$ is invertible on $D(X)$.
\end{lemma}
\begin{proof}
$f^*$ acts as the identity on $\widehat{\Pic}(X)$, hence $R(f^*)$ acts as $R(1)$, and this is not zero because the roots of $R$ all have absolute value $q$ or $q^{\frac12}$.
So it suffices to show that $R(f^*)$ is invertible on $C(X)_{\C}:=\left(\bigoplus_vC(X_v^{\an})\right)\otimes_{\R}\C$. Factor $R$ over $\C$ as 
\[ R(T)=a\prod_i\left(1-\frac{T}{\lambda_i}\right),\]
where $a\ne0$, and by lemma \ref{eigenvalues}, $|\lambda_i|$ is either $q^{\frac12}$ or $q$. $R(f^*)$ is invertible provided each term $1-f^*/\lambda_i$ is, and each term has inverse
\[ \left(1-\frac{f^*}{\lambda_i}\right)^{-1}=\sum_{k=0}^{\infty}\left(\frac{f^*}{\lambda_i}\right)^k,\]
provided this series converges absolutely with respect to the operator norm, which is defined with respect to the supremum norm $||\cdot||_{sup}$ on $C(X^{\an}_v)_{\C}$ for every place $v$. $f^*$ does not change the supremum norm, so the operator norm of $f^*$ is $1$, and 
\[ \left\lVert\left(\frac{f^*}{\lambda_i}\right)^k\right\lVert=\frac1{|\lambda_i|^k}\le q^{-\frac k2},\]
so the series converges absolutely.
\end{proof}

\begin{cor}
The exact sequence
\[ 0\longrightarrow D(X)\longrightarrow\widehat{\Pic}(X)_{\Q}'\longrightarrow\Pic(X)_{\Q}\longrightarrow0\]
has a unique $f^*$-equivariant splitting.
\end{cor}
\begin{proof}
Define 
\[ E(X):=\ker\left(R(f^*):\widehat{\Pic}(X)'_{\Q}\longrightarrow\widehat{\Pic}(X)'_{\Q}\right).\]
Since $R(f^*)$ kills all of $\Pic(X)_{\Q}$, this gives an $f^*$-invariant decomposition 
\[ \widehat{\Pic}(X)'_{\Q}=D(X)\bigoplus E(X)\]
such that the projection onto $\Pic(X)$ gives an isomorphism $E(X)\isom\Pic(X)_{\Q}$, whose inverse is the desired splitting.

We can write this down even more explicitly. For $M\in\Pic(X)_{\Q}$, let $\overline M$ be any choice of metric in $\widehat{\Pic}(X)'_{\Q}$. Then define
\[ \overline M_f:=\overline M-R(f^*)|_{D(X)}^{-1}R(f^*)\overline M.\]
\end{proof}

It now remains to show that this splitting satisfies (1) and (2). To start, suppose $M$ is in $\Pic^0(X)_{\Q}$. Let $x_0\in\Prep(f)$, then after replacing $f$ by an iterate and $K$ by a finite extension if necessary, we may assume that $x_0$ is a fixed point. Let $i:X\to \Alb(X)$ be the Albanese map taking $x_0\mapsto0$, then $f^*$ and $i^*$ induce the following commutative diagram, where $f':=(f^*)^{\vee}$:

\[
\begin{tikzcd}
\Pic^0(\Alb(X)) \arrow[shift left=.5ex]{r}{i^*}\arrow{d}[left]{(f')^*} & \Pic^0(X)\arrow[shift left=.5ex]{l}{\sim}\arrow{d}{f^*}\\
\Pic^0(\Alb(X)) \arrow[shift left=.5ex]{r}{i^*}\arrow{d}[left]{M\mapsto\overline M_{f'}} & \Pic^0(X)\arrow[shift left=.5ex]{l}{\sim}\arrow{d}{M\mapsto\overline M_f}\\
\widehat{\Pic}(\Alb(X))' \arrow{r}{i^*} & \widehat{\Pic}(X)'\\
\end{tikzcd}
\]

Because this commutes, it suffices to show (1) for abelian varieties, as $i^*$ takes $M_{f'}$ to $\overline M_f$, and the pullback of a flat metric is also flat. Now $[2]^*M=2M$, and since $[2]$ commutes with $f'$, 
\[ [2]^*\overline M_{f'}=2\overline M_{f'},\]
so that as in the proof of Lemma \ref{flatlemma}, we have that $\overline M_{f'}$, and hence also $\overline M_f$ is flat.

Finally, we show that (2) also holds. This is proven when $K$ is a number field in \cite[Theorem 4.9]{yz}, however the proof works identically in our geometric setting. 
\end{proof}

The above section also descends to a section $\trP{X}\to\widehat{\Pic}(X)\big/\imP{X},$ as by construction every element of $\imP{X}$ has a flat metric. Thus, we have an $f^*$-equivariant linear map
\[ \widehat{\ell}_f:\NS(X)_{\Q}\longrightarrow\left(\widehat{\Pic}(X)/\imP{X}\right)_{\Q}\]
given  by the composition of the section developed in Theorem \ref{admissible} and the map just preceding it.

\subsection{Rigidity of height zero points and preperiodic points}

Heights given by $f$-admissible metrized line bundles have particularly nice properties and correspond to the dynamical canonical heights defined by Call-Silverman \cite{CS}.
\begin{proposition}\label{canonicalheights}
Let $ M\in\Pic(X)_{\Q}$. Then:
\begin{enumerate}
\item
If $f^*M= M^\lambda$ for some $\lambda\in\Q$, then $f^*\overline M_f= \overline M_f^\lambda$ in $\widehat{\Pic}(X)_{\Q}$, and 
\[ h_{\overline M_f}(f(\cdot))=\lambda h_{\overline M_f}(\cdot).\]
\item
For $x\in\Prep(f)$, $\overline M_f|_x$ is trivial on $\widehat{\Pic}(x)_{\Q}$, and in particular $h_{\overline M_f}$ is zero on $\Prep(f)$.
\end{enumerate}
Further, if $M$ is ample and $f^*M=\lambda M$ for some $\lambda>1$ (in particular, if $M=L$), then:
\begin{enumerate}
\setcounter{enumi}{2}

\item
$h_{\overline M_f}(x)\ge0$ for all $x\in X(K)$, and
\item
if the above holds and $k$ is finite, $h_{\overline M_f}(x)=0$ if and only if $x\in\Prep(f)$.

\end{enumerate}
\end{proposition}
Call and Silverman~\cite{CS} establish that our height agrees with the dynamical canonical height $\widehat h_f$, and then the above all follow from well-known properties of dynamical heights proven in the same paper.

We can now state and prove our main theorem of this section.
\begin{theorem}\label{dynamicsthm}
Let $(f,L)$ and $(g,M)$ be two polarized algebraic dynamical systems on $X$. Define $Z_f:=\{x\in X(\overline K)|h_{\overline L_f}(x)=0\}$ to be the set of height zero points with respect to $\overline L_f$, and $Z_g$ the set of height zero points with respect to $\overline M_g$, and let $Z$ be the Zariski closure of $Z_f\cap Z_g$ in $X$. Then
\[ Z_f\cap Z(\overline K)=Z_g\cap Z(\overline K).\]
\end{theorem}

When $k$ is finite, $Z_f=\Prep(f)$ and $Z_g=\Prep(g)$, so Corollary \ref{dynamicsthmfinite} stated in the introduction follows as an immediate consequence. If $k$ is not finite, it is still true that $Z_f\supseteq\Prep(f)$, but there may be height zero points with infinite forward orbit. See Section~\ref{future} for further discussion.

\begin{proof} We begin by proving a simpler lemma, justifying the notation that $Z_f$ does not depend on the polarization $L$.

\begin{lemma}\label{multiplepolarizations}
Let $f:X\to X$, and let $L$ and $M$ be two ample line bundles which polarize $f$. Then
\[ \left\{x\in X(\overline K)\big|h_{\overline L_f}(x)=0\right\}\text{ is equal to }\left\{x\in X(\overline K)\big|h_{\overline M_f}(x)=0\right\},\]
and we unambiguously call both sets $Z_f$.
\end{lemma}

\begin{proof}
Since $L$ is ample, there exists a constant $c>0$ such that $cL-M$ is also ample. Then by Proposition~\ref{canonicalheights}, the canonical heights $h_{\overline M_f}$ and $h_{\overline {cL}_f}=ch_{\overline L_f}$ are related by
\[ 0\le h_{\overline M_f}(x)\le ch_{\overline L_f}(x)\]
for all $x\in X(\overline K)$. Thus 
\[ \left\{x\in X(\overline K)|h_{\overline L_f}(x)=0\right\}\subseteq\left\{x\in X(\overline K)|h_{\overline M_f}(x)=0\right\}.\]
By symmetry, we also have containment in the other direction.
\end{proof}

We now prove the theorem.

Let $Y$ be the normalization of an irreducible component of $Z$, assume $K$ is replaced by a finite extension if necessary so that $Y$ is geometrically normal, and say $\dim Y=d$. 
Let $\xi$ be the image of $L$ in $\NS(X)$. $\xi$ has two different lifts $\ell_f(\xi)$ and $\ell_g(\xi)$ to $\trP{X}_{\Q}$; let $L_f$ and $L_g$ be representatives in $\Pic(X)_{\Q}$ of these classes in $\trP{X}_{\Q}$. Since $L$ is one such choice of representative for $\ell_f(\xi)$ and ampleness is preserved by numerical equivalency, $L_f$ and $L_g$ must both be ample. 

By Theorem \ref{admissible}, $L_f$ and $L_g$ have $f$- and respectively $g$-admissible metrics, which we call $\overline L_f$ and $\overline L_g$. Both are nef. Their sum $\overline N:=\overline L_f+\overline L_g$ is also nef, and defines a height function $h_{\overline N}$, which does not depend on the choice of representatives of cosets modulo the trace.

 By Lemma~\ref{multiplepolarizations} and the premise that $Z_f\cap Z_g\cap Z(\overline K)$ is dense, $Y$ has a dense set of points which have height zero under $h_{\overline N}$.  
By the successive minima (Proposition~\ref{successiveminima}),
\[ \lambda_1(Y,{\overline N})=h_{\overline N}(Y)=0.\]

Rewriting the height of $Y$ in terms of intersections,
\[ 0=\left(\overline L_f|_Y+\overline L_g|_Y\right)^{d+1}=\sum_{i=0}^{d+1}{d+1 \choose i}\left(\overline L_f|_Y\right)^i\cdot\left(\overline L_g|_Y\right)^{d+1-i}.\]

Since both $\overline L_f$ and $\overline L_g$ are nef, every term in the sum on the right is non-negative, hence all must be zero. Then
\[ \left(\overline L_f|_Y-\overline L_g|_Y\right)^2\cdot\left(\overline L_f|_Y+\overline L_g|_Y\right)^{d-1}=0,\] 
as well. 
Because $L_f-L_g$ is zero in the N\'eron-Severi group, and thus numerically trivial we also have
\[ \left(L_f|_Y-L_g|_Y\right)\cdot\left(L_f|_Y+L_g|_Y\right)^{d-1}=0.\]

Additionally, $(\overline L_f-\overline L_g)$ is clearly $(\overline L_f+\overline L_g)$-bounded, and we are nearly in the right setting to apply Theorem \ref{hodgeindex}, except that $(\overline L_f+\overline L_g)$ is nef, but not necessarily arithmetically positive. 

To fix this, we simply adjust the metric by a small positive factor: let $\overline C\in\widehat\Pic(K)$ with $\widehat{\deg}(\overline C)>0$. Replace the pair $\left(\overline L_f-\overline L_g,\overline L_f+\overline L_g\right)$ by $\left(\overline L_f+\overline L_g,\overline L_f+\overline L_g+\overline \pi^*C\right)$. Since $L_f-L_g$ is numerically trivial, the metric on $\overline L_f-\overline L_g$ is flat, so adding $\overline \pi^*C$, which is vertical, does not change the intersection number.
All the conditions of the theorem are now satisfied, so that the theorem tells us
\[ \left(\overline L_f-\overline L_g\right)\in\widehat{\Pic}(K)_{\Q}+\imP{X}_{\Q}.\]

We therefore conclude by Theorem \ref{numericallytrivial} that $h_{\overline L_f}-h_{\overline L_g}$ is a constant height function on $Y$. Since these two heights both take value zero on a dense set in $Z$, they must be equal on $Y$. Thus these heights define the same sets of height zero points, and then by Lemma~\ref{multiplepolarizations}, $Z_f$ and $Z_g$ agree on $Y$, and hence on all of $Z$.
\end{proof}


\section{Related results and further questions}\label{future}

\subsection{Rigidity of preperiodic points over global function fields}
We first summarize some basic consequences of Theorem~\ref{dynamicsthm} when $K$ is a global function field, particularly in the case when $\Prep(f)\cap\Prep(g)$ is dense in $X$.

\begin{lemma}
Let $K$ be a global function field, and let $f$ and $g$ be two polarized algebraic dynamical systems on a projective variety $X$. Then the following are equivalent:
\begin{enumerate}
\item
$\Prep(f)=\Prep(g)$.
\item
$\Prep(f)\cap\Prep(g)$ is dense in $X$.
\item
$\Prep(f)\subset\Prep(g)$.
\item
$g(\Prep(f))\subset \Prep(f)$.
\end{enumerate}
\end{lemma}

\begin{proof}
The equivalence of $(1)$ and $(2)$ is an immediate consequence of Theorem~\ref{dynamicsthm} and the fact that over a global function field, all dynamical height zero points are preperiodic. Clearly $(1)$ implies $(4)$. By Fakhruddin~\cite{Fakhruddin}, $\Prep(f)$ is always dense in $X$, hence $(3)$ implies $(2)$. We now show $(4)$ implies $(3)$.

Stratify $\Prep(f)$ by degree, writing
\[ \Prep(f)=\bigcup_{d\ge0}\Prep(f,d),\]
where
\[ \Prep(f,d):=\left\{x\in\Prep(f)|[K(x):K]\le d\right\}.\]
Since each $\Prep(f,d)$ has height zero and bounded degree, it is finite. Now $(4)$ says that $g$ fixes $\Prep(f)$, but since $g$ is defined over $K$, it fixes each $\Prep(f,d)$ as well. Thus every point of $\Prep(f)$ has finite forward orbit under $g$.
\end{proof}

This lemma suggests two related questions which we do not answer here.
\begin{enumerate}
\item
When is $\Prep(f)$ equal to $\Prep(g)$?
\item
If $\Prep(f)=\Prep(g)$, how closely related must $f$ and $g$ be?
\end{enumerate}
In the case of $f:\P^1\to\P^1$, Mimar~\cite{mimar} gives a variety of partial answers to these questions, with the general implication being that if $f$ and $g$ have the same preperiodic points, their Julia sets must also be very similar. But this is likely very difficult in dimension greater than one.

\subsection{Preperiodic points over larger function fields}
Theorem~\ref{hodgeindex} and most of the proof of Theorem~\ref{dynamicsthm} hold over all transcendence degree one function fields, not just global function fields. But because the Northcott principal fails when $k$ is not a finite field or the algebraic closure of a finite field, we cannot equate height zero points with preperiodic points over arbitrary function fields, and thus Theorem~\ref{dynamicsthm} is a statement about height zero points and not preperiodic points. In this broader setting, however, some things can still be said.

Baker~\cite{bakerisotrivial} proves the following theorem, first proven by Benedetto~\cite{benedetto} in the case of polynomials.

\begin{theorem}
Let $f:\P^1_K\to\P^1_K$ be a rational function of degree $\ge2$, and suppose that $f$ is not isotrivial, in the sense that there exists no finite extension $K'$ of $K$ and M\"obius transformation $M\in \PGL_2(K')$ such that 
\[ f':=M^{-1}\circ f\circ M\]
is defined over $k$. Then
\[ \Prep(f)=Z_f.\]
\end{theorem}

Thus Theorem~\ref{dynamicsthm} proven here immediately implies Corollary~\ref{bakercor}.

In higher dimension isotriviality is less straightforward to classify. When $A$ is an abelian variety, its $K/k$-trace classifies how isotrivial it is, and then the Lang-N\'eron theorem provides a Northcott-like result for the N\'eron-Tate canonical height (the dynamical height induced by $[n]$): height zero points fall into only finitely many cosets of $\Tr_{K/k}(A)(k)\hookrightarrow A(K)$. 

There is no notion of a trace for general varieties, however, and $\iota^{-1}\Tr_{K/k}(\Alb(X))$ is not a sufficient substitute, as $\Alb(X)$ will often be trivial. Chatzidakis and Hrushovski~\cite{chatzidakis1,chatzidakis2} instead use model theory, and a variant of isotriviality called \emph{constructible descent} to $k$. Their theorem generalizes both Baker's result and the Lang-N\'eron Theorem.

\begin{theorem}
Let $K$ be any function field and let $k$ be its field of constants. Let $f:X\to X$ be an algebraic dynamical system defined over $K$, and assume $f$ does not constructibly descend to $k$. Then for every point $x\in X(\overline K)$ with dynamical height zero there exists a proper Zariski closed subset $Y_x\subsetneq X$ such that the orbit of $x$ is contained in $Y_x$.
\end{theorem}

The author is optimistic that the methods of arithmetic heights and rigidity theorem of this paper, combined with model-theoretic treatment of isotriviality will yield stronger dynamics results over general function fields in the future.


\end{document}